\documentclass[11pt]{amsart}

\scrollmode

\usepackage{times}
\usepackage{amsmath,graphicx}
\usepackage{latexsym}
\usepackage{amscd,amsthm,amssymb}
\usepackage[all]{xy}
\usepackage{microtype}

\newtheorem{thm}{Theorem}[section]
\newtheorem*{thm*}{Theorem} 
\newtheorem{prop}[thm]{Proposition}
\newtheorem{lem}[thm]{Lemma}
\newtheorem{cor}[thm]{Corollary}
\newtheorem*{cor*}{Corollary}

\newtheorem{conj}[thm]{Conjecture}

\theoremstyle{definition}
\newtheorem{ex}[thm]{Example}
\newtheorem{defn}[thm]{Definition}

\theoremstyle{remark}
\newtheorem{rem}[thm]{Remark}

\numberwithin{equation}{section}

\title{Cyclic group actions and embedded spheres in 4-manifolds}
\author{M.~J.~D.~Hamilton}
\address{      Institute for Geometry and Topology\\
               University of Stuttgart\\
               Pfaffenwaldring 57\\
               70569 Stuttgart\\
               Germany}
\email{mark.hamilton@math.lmu.de}
\date{\today}
\begin{document}

\begin{abstract}In this note we derive an upper bound on the number of 2-spheres in the fixed point set of a smooth and homologically trivial cyclic group action of prime order on a simply-connected 4-manifold. This improves the a priori bound which is given by one half of the Euler characteristic of the 4-manifold. The result also shows that in some cases the 4-manifold does not admit such actions of a certain order at all or that any such action has to be pseudofree.
\end{abstract}

\maketitle

\section{Introduction}
Actions of finite groups, in particular cyclic groups $\mathbb{Z}_p$ of prime order $p$, on simply-connected 4-manifolds have been studied in numerous places in the literature. An interesting subclass are those actions which act trivially on homology. In the topological setting, Edmonds has shown \cite[Theorem 6.4]{Ed1} that every closed, simply-connected, topological 4-manifold admits for every $p>3$ a (non-trivial) homologically trivial action which is {\em locally linear}. However, it is an open question from the Kirby list if such actions exist in the {\em smooth} setting for 4-manifolds like the $K3$ surface (it is known that there is no such action of $\mathbb{Z}_2$ \cite{Mat, Rub} on $K3$ and no such action of $\mathbb{Z}_p$ which is holomorphic \cite{BuRa, Pe} or symplectic \cite{ChKw}).

The actions in the theorem of Edmonds can be assumed to be pseudofree, i.e.~the fixed point set consists of isolated points. In general, if the action is homologically trivial, the fixed point set will consist of isolated points and disjoint embedded 2-spheres. We recall this fact in Proposition \ref{main prop}. If $m$ is the number of points and $n$ the number of spheres, then $m+2n$ is equal to the Euler characteristic $\chi(M)$ of the 4-manifold. This implies an a priori upper bound on the number of spheres:
\begin{equation*}
n\leq\frac{\chi(M)}{2}.
\end{equation*}
A natural question is whether all cases of possible values for $n$ can occur. We will show that this upper bound can indeed be improved, for example, by a factor of roughly $\frac{1}{2}$ if the 4-manifold $M$ and the action are smooth, $M$ is smoothly minimal and the Seiberg-Witten invariants of $M$ are non-vanishing. More precisely we will show the following: We say that a 4-manifold $M$ satisfies property $(\ast)$ if every smoothly embedded 2-sphere in $M$ that represents a non-zero rational homology class has negative self-intersection. For example, a 4-manifold $M$ with $b_2^+(M)>1$ and non-vanishing Seiberg-Witten invariants satisfies property $(\ast)$. Then we have (cf.~Corollary \ref{main cor to thm}):
\begin{cor*}
Let the group $\mathbb{Z}_p$ act homologically trivially and smoothly on a simply-connected, smooth 4-manifold $M$ that satisfies property $(\ast)$. Then
\begin{equation*}
n\leq\frac{p\chi(M)-c_1^2(M)}{3(p-1)}.
\end{equation*}
If in addition $M$ is smoothly minimal, then
\begin{equation*}
n\leq\frac{p\chi(M)-c_1^2(M)}{2(2p-1)}.
\end{equation*}
Independently of $c_1^2(M)$ we have in these cases the bounds
\begin{equation*}
n<  \frac{\chi(M)}{3}\left(1+\frac{2}{p-1}\right)
\end{equation*}
and
\begin{equation*}
n<  \frac{\chi(M)}{4}\left(1+\frac{3}{2p-1}\right),
\end{equation*}
respectively.
\end{cor*}
The proof uses the $G$-signature theorem together with an estimate on the signature defects at the fixed points. Even though the proof is elementary, it seems worthwhile to record this fact together with a number of corollaries, in particular in the situation that the theorem implies $n<0$ (no action possible) or $0\leq n<1$ (every action is pseudofree). 

This result has applications especially for smooth, homologically trivial $\mathbb{Z}_2$- and $\mathbb{Z}_3$-actions on general, smooth, simply-connected 4-manifolds as well as for $\mathbb{Z}_p$-actions on possible examples of exotic smooth 4-manifolds homeomorphic to $S^2\times S^2$ or $\mathbb{CP}^2\#{\overline{\mathbb{CP}}{}^{2}}$. Three consequences of the main theorem are Corollary \ref{cor edmonds smooth invol}, Corollary \ref{sign=-1} and Corollary \ref{cor Z2 sign 0 -1} that lead to implications in particular for smooth involutions. The first corollary is related to a special case of the problem from the Kirby list and implies that a simply-connected, {\em non-spin} 4-manifold with {\em positive signature} that satisfies property $(\ast)$ does not admit homologically trivial, {\em smooth} involutions. The same is true according to the second corollary if the signature is equal to $-1$ and the manifold is in addition smoothly minimal. Both results are a partial extension of a theorem of Ruberman for spin 4-manifolds to the non-spin case and contrast a theorem of Edmonds, who has shown that every smooth, simply-connected, non-spin 4-manifold admits a homologically trivial, {\em locally linear} involution. The third corollary implies that a homologically trivial, {\em smooth} involution on a simply-connected 4-manifold that satisfies property $(\ast)$ and has vanishing signature (this is by Ruberman's theorem the case if the 4-manifold is spin, for example) is necessarily pseudofree, i.e.~the fixed point set consists of a collection of isolated points.

\subsection*{Convention} All 4-manifolds in the following will be closed, oriented and connected and have $b_2(M)>0$. All spheres embedded in 4-manifolds will be 2-dimensional. All group actions will be non-trivial and orientation-preserving.

\subsection*{Acknowledgements} I would like to thank Dieter Kotschick for helpful comments regarding reference \cite{Ko} and an anonymous referee for valuable suggestions to improve the quality of the paper.

\section{Spheres in the fixed point set and the $G$-signature theorem}
Let $M$ denote a simply-connected, topological 4-manifold with a {\em locally linear} action of a cyclic group $G=\mathbb{Z}_p$, with $p\geq 2$ a prime. The group action is generated by a locally linear homeomorphism $\tau\colon M\rightarrow M$ of order $p$, such that $\tau$ is not equal to the identity. There is an induced action of $G$ on $H^2(M;\mathbb{Z})$ preserving the intersection form. According to \cite{Ed, KwSch} this action decomposes over the integers into $t$ copies of the trivial action of rank 1, $c$ copies of the cyclotomic action of rank $p-1$ and $r$ copies of the regular action of rank $p$, where $t,c,r$ are certain non-negative integers. As a consequence, the second Betti number of $M$ is equal to
\begin{equation*}
b_2(M)=t+c(p-1)+rp.
\end{equation*}
In particular we have:
\begin{lem}\label{lem p b_2}
If $p>b_2(M)+1$, then $G$ acts trivially on homology.
\end{lem}
Let $F$ denote the fixed point set of the locally linear homeomorphism $\tau$. Since $G$ is of prime order, the set $F$ is the fixed point set of every group element in $G$ different from the identity. The fixed point set $F$ is a closed topological submanifold of $M$ \cite[p.~171]{Br}. The action is locally linear and hence given by an orthogonal action in a neighbourhood of a fixed point. Since the action preserves orientation, the fixed point set $F$ has even codimension \cite{Smith}. It consists of a disjoint union of finitely many isolated points and finitely many closed surfaces. If $p$ is odd, then every surface in the fixed point set is orientable \cite[p.~175]{Br}.

The next lemma follows from \cite[Proposition 2.5]{Ed}:
\begin{lem}Suppose that the fixed point set $F$ has more than one component. Then every surface component of $F$ represents a non-zero class in $H_2(M;\mathbb{Z}_p)$.
\end{lem}
If the action is not free, then according to \cite[Proposition 2.4]{Ed} the $\mathbb{Z}_p$-Betti numbers of the fixed point $F$ satisfy
\begin{align*}
b_1(F;\mathbb{Z}_p)&=c\\
b_0(F;\mathbb{Z}_p)+b_2(F;\mathbb{Z}_p)&=t+2.
\end{align*}
Let $\chi(M)=b_2(M)+2$ denote the Euler characteristic of $M$. If $G$ acts trivially on homology, then $\chi(F)=\chi(M)$ by the Lefschetz fixed point theorem. Hence the action is not free and we get:
\begin{prop}\label{main prop}
Suppose that $G$ acts trivially on the homology of $M$. Then $F$ consists of a disjoint union of $m$ isolated points and $n$ spheres, with $m+2n=\chi(M)$. Since $b_2(M)>0$, after a choice of orientation, every sphere in $F$ represents a non-zero class in $H_2(M;\mathbb{Z})$.
\end{prop}
From now on we assume that the action of $G$ is {\em trivial on homology}. We want to improve the upper bound $\frac{1}{2}\chi(M)$ on the number $n$ of spheres. We can use the $G$-signature theorem \cite{AtSing}, which is valid not only for smooth, but also for locally linear actions in dimension 4, cf.~\cite{Wall} and a remark in \cite[p.~164]{Ed1} (all our applications will be for smooth actions). Let $S_1,\ldots,S_n$ denote the spherical components of the fixed point set $F$ and $P$ the set of isolated fixed points. Note that the signature satisfies
\begin{equation*}
\mathrm{sign}(M/G)=\mathrm{sign}(M),
\end{equation*}
since the action of $G$ is trivial on homology. The $G$-signature theorem implies \cite[p.~14--17]{Hirz}:
\begin{equation*}
(p-1)\mathrm{sign}(M)=\sum_{x\in P}\mathrm{def}_x+\frac{p^2-1}{3}\sum_{i=1}^n[S_i]^2.
\end{equation*}
Here $[S_i]^2$ denotes the self-intersection number of the sphere $S_i$. The numbers $\text{def}_x$ are equal, in Hirzebruch's notation, to $\text{def}(p;q,1)$ for certain integers $q$ coprime to $p$ and depending on $x$. We have
\begin{equation*}
\mathrm{def}(p;q,1)=-\frac{2}{3}(q,p)=-4p\sum_{k=0}^{p-1}\left(\left(\frac{k}{p}\right)\right)\left(\left(\frac{qk}{p}\right)\right).
\end{equation*}
In this equation $(q,p)$ denotes the Dedekind symbol, while $((\cdot))\colon \mathbb{R}\rightarrow\mathbb{R}$ is a certain function introduced by Rademacher and given by
\begin{align*}
((z))&=z-[z]-\frac{1}{2},\quad\text{if $z$ is not an integer}\\
((z))&=0,\quad\text{if $z$ is an integer}.
\end{align*}
Here $[z]$ denotes the greatest integer less than or equal to $z$. We want to prove the following estimate:
\begin{lem}\label{main lem} For all prime numbers $p$ and integers $q$ coprime to $p$ we have
\begin{equation*}
|\mathrm{def}(p;q,1)|\leq |\mathrm{def}(p;1,1)|=\frac{1}{3}(p-1)(p-2).
\end{equation*}
\end{lem}
\begin{proof}
We have by Cauchy-Schwarz
\begin{align*}
\left|\sum_{k=0}^{p-1}\left(\left(\frac{k}{p}\right)\right)\left(\left(\frac{qk}{p}\right)\right)\right|&\leq \left(\sum_{k=1}^{p-1}\left(\left(\frac{k}{p}\right)\right)^2\right)^{\frac{1}{2}}\cdot\left(\sum_{k=1}^{p-1}\left(\left(\frac{qk}{p}\right)\right)^2\right)^{\frac{1}{2}}\\
&= \sum_{k=1}^{p-1}\left(\left(\frac{k}{p}\right)\right)^2,
\end{align*}
because $q$ generates $\mathbb{Z}_p$ and $((0))=0$. Since $0<\frac{k}{p}<1$ for all $k=1,\ldots,p-1$ we have
\begin{align*}
\sum_{k=1}^{p-1}\left(\left(\frac{k}{p}\right)\right)^2&=\sum_{k=1}^{p-1}\left(\frac{k}{p}-\frac{1}{2}\right)^2\\
&=\sum_{k=1}^{p-1}\left(\frac{k^2}{p^2}-\frac{k}{p}+\frac{1}{4}\right)\\
&=\frac{1}{6p^2}(p-1)p(2p-1)-\frac{1}{2p}(p-1)p+\frac{p-1}{4}\\
&=\frac{1}{6p}(2p^2-3p+1)-\frac{1}{2p}(p^2-p)+\frac{p-1}{4}\\
&=\frac{1}{12p}(4p^2-6p+2-6p^2+6p+3p^2-3p)\\
&=\frac{1}{12p}(p^2-3p+2)\\
&=\frac{1}{12p}(p-1)(p-2).
\end{align*}
This implies the claim. The number $\mathrm{def}(p;1,1)$ has also been calculated in equation (28) in \cite{Hirz}.
\end{proof}

We can now prove the main theorem. We use the standard notation 
\begin{equation*}
c_1^2(M)=2\chi(M)+3\mathrm{sign}(M)
\end{equation*}
for every 4-manifold $M$. We abbreviate the following conditions on the action and the manifold by simply saying that "$\mathbb{Z}_p$ acts homologically trivially on a simply-connected 4-manifold $M$":
\begin{quote}
The group $\mathbb{Z}_p$, with $p\geq 2$ prime, acts locally linearly and homologically trivially on a simply-connected, topological 4-manifold $M$.
\end{quote}
We consider in the following only actions of this kind.
\begin{thm}\label{main thm} Let $\mathbb{Z}_p$ act homologically trivially on a simply-connected 4-manifold $M$. Suppose that all spheres $S$ in the fixed point set of the action satisfy an a priori bound $[S]^2\leq s< 0$ for some integer $s$. Then the number $n$ of spheres in the fixed point set satisfies the upper bound
\begin{equation*}
n\leq\frac{p\chi(M)-c_1^2(M)}{p(2-s)-(4+s)}.
\end{equation*}
For all possible values of $c_1^2(M)$ we have the bound
\begin{equation*}
n< \frac{\chi(M)}{2-s}\left(1+\frac{6}{p(2-s)-(4+s)}\right).
\end{equation*}
\end{thm}
\begin{proof}
By Proposition \ref{main prop} the number of isolated fixed points in $F$ is $\chi(M)-2n$. By the $G$-signature theorem and Lemma \ref{main lem} we have
\begin{equation*}
(p-1)\mathrm{sign}(M)\leq \frac{1}{3}(p-1)(p-2)(\chi(M)-2n)+\frac{1}{3}sn(p^2-1).
\end{equation*}
This implies the first claim (note that the denominator is positive under our assumption $s< 0$). The second claim follows from the estimate $\mathrm{sign}(M)>-\chi(M)$, which is true for all oriented 4-manifolds with $b_1(M)=0$. 
\end{proof}

\section{Smoothly embedded spheres}\label{sect emb spheres}

\begin{defn}We say that a smooth 4-manifold $M$ satisfies {\bf property} $\mathbf{(\ast)}$ if the following holds:
\begin{quote}
Every smoothly embedded sphere $S$ in $M$ that represents a non-zero homology class $[S]\in H_2(M;\mathbb{Q})$ has negative self-intersection number.
\end{quote}
\end{defn}
We are interested under which conditions a 4-manifold $M$ satisfies property $(\ast)$. The following is clear:
\begin{prop}
Let $M$ be a smooth 4-manifold. Assume that $b_2^+(M)=0$. Then $M$ satisfies property $(\ast)$.
\end{prop}
The next theorem is well-known, cf.~\cite[Proposition 1]{Ko}. The statement also follows from the adjunction inequality \cite{KrMr, GS}. 
\begin{prop}\label{prop sphere adjunction} Let $M$ be a smooth 4-manifold. Assume that $b_2^+(M)>1$ and the Seiberg-Witten invariants of $M$ do not vanish identically. Then $M$ satisfies property $(\ast)$.
\end{prop}
We did not find in the literature a similarly general theorem in the case of 4-manifolds $M$ with $b_2^+(M)=1$. To describe what we can show in this case, recall that a {\em rational surface} is a smooth 4-manifold diffeomorphic to $S^2\times S^2, \mathbb{CP}^2$ or $\mathbb{CP}^2\#n{\overline{\mathbb{CP}}{}^{2}}$ with $n\geq 1$, while a {\em ruled surface} is an oriented $S^2$-bundles over a Riemann surface $\Sigma_g$ of genus $g\geq 0$ (there exist up to diffeomorphism two such ruled surfaces for every genus $g$. The ruled surface is called {\em irrational} if $g\geq 1$.) We can then prove the following:
\begin{prop}\label{prop sphere adjunction b+=1} Let $M$ be a smooth 4-manifold. Assume that $b_2^+(M)=1$, $b_2^-(M)\leq 9$, $b_1(M)=0$ and $M$ is not diffeomorphic to a rational surface. If $M$ admits a symplectic form, then $M$ satisfies property $(\ast)$.
\end{prop}
\begin{rem}
In this situation, the assumption $b_2^-(M)\leq 9$ is equivalent to $K^2\geq 0$, where $K$ denotes the canonical class of the symplectic form, because $K^2=2\chi(M)+3\mathrm{sign}(M)$.
\end{rem}
For the proof recall the following theorem of Liu \cite[Theorem B]{Liu} (slightly adapted to make the statement more precise):
\begin{thm}[Liu]\label{thm liu} Let $M$ be a symplectic 4-manifold with $b_2^+(M)=1$. If $K\cdot\omega<0$, then $M$ must be either rational or (a blow-up of) an irrational ruled 4-manifold.
\end{thm}
We also need an adjunction inequality of Li and Liu \cite[p.~467]{LiLiu}:
\begin{thm}[Li-Liu]\label{thm li liu}
Suppose $M$ is a symplectic 4-manifold with $b_2^+(M)=1$ and $\omega$ is the symplectic form. Let $C$ be a smooth, connected, embedded surface with non-negative self-intersection. If $[C]\cdot\omega>0$, then the genus of $C$ satisfies $2g(C)-2\geq K\cdot[C]+[C]^2$.
\end{thm}
We have the following general light cone lemma, compare with \cite[Lemma 2.6]{LiLiu}:
\begin{lem}\label{lem light}
Let $M$ be a 4-manifold with $b_2^+(M)=1$. The forward cone is one of the two connected components of $\{a\in H^2(M;\mathbb{R})\mid a^2>0\}$. Then the following holds for all elements $a,b\in H^2(M;\mathbb{R})$:
\begin{enumerate}
\item If $a$ is in the forward cone and $b$ in the closure of the forward cone with $b\neq 0$, then $a\cdot b>0$.
\item If $a$ and $b$ are in the closure of the forward cone, then $a\cdot b\geq 0$.
\item If $a$ is in the forward cone and $b$ satisfies $b^2\geq 0$ and $a\cdot b\geq 0$, then $b$ is in the closure of the forward cone.
\end{enumerate}
\end{lem}
\begin{proof}
With respect to a suitable basis of the vector space $H^2(M;\mathbb{R})$ we have $a\cdot b=a_0b_0-\sum a_ib_i$ where the elements $a$ in the forward cone satisfy $a_0>0$. Then (a) and (b) follow by applying the Cauchy-Schwarz inequality:
\begin{equation*}
\sum a_ib_i\leq \sqrt{\sum a_i^2}\sqrt{\sum b_i^2}.
\end{equation*}
For (c) assume by contradiction $b_0<0$. Then the vector $-b$ is in the closure of the forward cone, so (a) implies $a\cdot (-b)>0$ and hence $a\cdot b<0$, a contradiction.
\end{proof}
We can now prove Proposition \ref{prop sphere adjunction b+=1}:
\begin{proof}
Let the forward cone be defined by the class of $\omega$. Our assumptions together with Theorem \ref{thm liu} and Lemma \ref{lem light} imply that the canonical class $K$ is in the closure of the forward cone. Suppose that the class $[S]$ of a sphere $S$ satisfies $[S]\neq 0$ and $[S]^2\geq 0$. Choose the orientation on $S$ such that $[S]$ is in the closure of the forward cone. By Lemma \ref{lem light}, $[S]\cdot \omega>0$. Then Theorem \ref{thm li liu} applies and shows that $-2\geq K\cdot[S]+[S]^2$. However, Lemma \ref{lem light} implies that $K\cdot[S]\geq 0$. This is a contradiction.
\end{proof}
We conjecture the following:
\begin{conj} Let $M$ be a smooth 4-manifold. Assume that $b_2^+(M)=1$, $b_2^-(M)\leq 9$, $H_1(M;\mathbb{Z})=0$ and $M$ has non-trivial small perturbation Seiberg-Witten invariants. Then $M$ satisfies property $(\ast)$. 
\end{conj}
For a definition of the small perturbation Seiberg-Witten invariants see \cite{Sz}. 

\section{The main corollary for smooth actions}\label{sect main cors}

Recall that an oriented 4-manifold is called {\em (smoothly) minimal} if it does not contain smoothly embedded spheres of self-intersection $-1$.
\begin{cor}\label{main cor to thm}
Let the group $\mathbb{Z}_p$ act homologically trivially and smoothly on a simply-connected, smooth 4-manifold $M$ that satisfies property $(\ast)$. Then
\begin{equation*}
n\leq\frac{p\chi(M)-c_1^2(M)}{3(p-1)}.
\end{equation*}
If in addition $M$ is smoothly minimal, then
\begin{equation*}
n\leq\frac{p\chi(M)-c_1^2(M)}{2(2p-1)}.
\end{equation*}
Independently of $c_1^2(M)$ we have in these cases the bounds
\begin{equation*}
n<  \frac{\chi(M)}{3}\left(1+\frac{2}{p-1}\right)
\end{equation*}
and
\begin{equation*}
n<  \frac{\chi(M)}{4}\left(1+\frac{3}{2p-1}\right),
\end{equation*}
respectively.
\end{cor}
\begin{proof}
If the action is smooth, then every sphere in $F$ is smoothly embedded \cite[p.~309]{Br}. The first claim follows with Theorem \ref{main thm}, since $[S]^2\leq -1$ for every embedded sphere $S$ representing a non-zero homology class if $M$ satisfies property $(\ast)$. If $M$ is smoothly minimal, spheres of self-intersection $-1$ do not exist in $M$, hence $[S]^2\leq -2$.
\end{proof}
This improves the a priori bound $n\leq \frac{1}{2}\chi(M)$ by a factor of approximately $\frac{2}{3}$ and $\frac{1}{2}$, at least for large $p$.
\begin{ex} Let $M=E(k)_{a,b}$ be a simply-connected, minimal elliptic surface with multiple fibres of coprime indices $a,b$. Assume that either $k\geq 2$, or $k=1$ and both $a,b\neq 1$. Then $M$ is smoothly minimal, symplectic and irrational and thus satisfies property $(\ast)$. We have $c_1^2(M)=0$ and $\chi(M)=12k$. Therefore
\begin{equation*}
n\leq 3k\left(1+\frac{1}{2p-1}\right).
\end{equation*}
This rules out some of the possible $\mathbb{Z}_3$-actions on elliptic surfaces in \cite{LiH}.
\end{ex}

\section{The case $n<0$: non-existence of actions}

Since the integer $n$ has to be non-negative if an action exists, we get:
\begin{prop}\label{prop nonexistence} Let the group $\mathbb{Z}_p$ act homologically trivially on a simply-connected 4-manifold $M$. Suppose that all spheres $S$ in the fixed point set of the action satisfy an a priori bound $[S]^2\leq s< 0$ for some integer $s$. Then
\begin{equation*}
p\chi(M)\geq c_1^2(M).
\end{equation*} 
\end{prop}

\begin{cor}\label{cor p23 non} Let the group $\mathbb{Z}_p$ act homologically trivially and smoothly on a simply-connected, smooth 4-manifold $M$ that satisfies property $(\ast)$. If $p=2$, then $\mathrm{sign}(M)\leq 0$. If $p=3$, then $c_1^2(M)\leq 3\chi(M)$.
\end{cor}
\begin{rem} Ruberman \cite{Rub} has shown that if $\mathbb{Z}_2$ acts homologically trivially and {\em locally linearly} on a simply-connected {\em spin} 4-manifold, then $\mathrm{sign}(M)=0$. The first part of Corollary \ref{cor p23 non} is a partial extension of this result to {\em smooth} $\mathbb{Z}_2$-actions on {\em non-spin} 4-manifolds. Regarding the second statement, it is not known if there exist simply-connected, smooth 4-manifolds with non-trivial Seiberg-Witten invariants and $c_1^2(M)> 3\chi(M)$ (for more on this question see \cite[Section 10.3]{GS}). Note that any simply-connected 4-manifold satisfies a priori $c_1^2(M)< 5\chi(M)$.
\end{rem}
A non-singular, odd, integral, bilinear form $Q$ on a finitely generated free abelian group $V$ is said to have {\em characteristic signature} if there exists an indivisible characteristic element $v\in V$ such that $v\cdot v=\mathrm{sign}(Q)$. The intersection forms of {\em smooth}, simply-connected, non-spin 4-manifolds are direct sums of copies of the forms $(+1)$ and $(-1)$ (this is clear in the indefinite case and follows in the definite case by Donaldson's theorem \cite{Don}) and hence are always characteristic. The next theorem of Edmonds then follows from \cite[Corollary 11]{Ed2}:
\begin{thm}[Edmonds]\label{Edmonds exist loc lin invol} Every smooth, simply-connected, non-spin 4-manifold $M$ admits a homologically trivial, locally linear involution whose fixed point set consists of a single sphere $S$ with $[S]^2=\mathrm{sign}(M)$ and a collection of isolated points.
\end{thm}
By contrast, the following corollary is implied by Corollary \ref{cor p23 non}:
\begin{cor}\label{cor edmonds smooth invol} Let $M$ be a smooth, simply-connected 4-manifold $M$ that satisfies property $(\ast)$ and has positive signature. Then $M$ does not admit a homologically trivial, smooth involution.
\end{cor}
This corollary is relevant only if $M$ is non-spin because of Ruberman's theorem.
\begin{ex}Let $M$ be a simply-connected, complex algebraic surface of general type and positive signature with $b_2^+(M)>1$ (see e.g.~\cite{PPX} and the references therein for the construction of such surfaces). Then $M$ satisfies property $(\ast)$ by Proposition \ref{prop sphere adjunction} and the non-triviality of the Seiberg-Witten invariants for surfaces of general type \cite{Wi}. Hence $M$ does not admit a homologically trivial, smooth $\mathbb{Z}_2$-action. However, if $M$ is non-spin (for example, if $M$ is a blow-up of a spin surface of general type), then it admits a homologically trivial, locally linear $\mathbb{Z}_2$-action by Theorem \ref{Edmonds exist loc lin invol} of Edmonds.
\end{ex}
We can also prove the following:
\begin{cor}\label{sign=-1}
Let $M$ be a simply-connected, smooth, minimal 4-manifold with $\mathrm{sign}(M)=-1$ that satisfies property $(\ast)$. Then $M$ does not admit a homologically trivial, smooth involution.
\end{cor}
\begin{proof}
Since $\mathrm{def}_x=0$ for isolated fixed points of involutions, the $G$-signature theorem implies for such an action
\begin{equation*}
\mathrm{sign}(M)=\sum_{i=1}^n[S_i]^2.
\end{equation*}
This cannot be satisfied, because $\mathrm{sign}(M)=-1$ and $[S_i]^2\leq -2$ under our assumptions.
\end{proof}
\begin{rem}
Note that such a manifold is always non-spin according to Rohlin's theorem. The proof of Corollary \ref{sign=-1} also gives a further explanation for the result in Corollary \ref{cor edmonds smooth invol}.
\end{rem}

\section{The case $0\leq n<1$: action is pseudofree}

We can also study the case $0\leq n<1$. This will elucidate the situation close to or on the boundary of the allowed regions given by Proposition \ref{prop nonexistence} and Corollary \ref{cor p23 non}.
\begin{prop}\label{prop pseudofree} Let the group $\mathbb{Z}_p$ act homologically trivially on a simply-connected 4-manifold $M$. Suppose that all spheres $S$ in the fixed point set of the action satisfy an a priori bound $[S]^2\leq s< 0$ for some integer $s$ and that $M$ satisfies
\begin{equation*}
p\chi(M)-c_1^2(M)<p(2-s)-(4+s).
\end{equation*}
Then $n=0$, hence the fixed point set consists only of isolated points, i.e.~the action is pseudofree.
\end{prop}
The following is an application to involutions on 4-manifolds with $\mathrm{sign}(M)=0$:
\begin{cor}\label{cor Z2 sign 0 -1}Let the group $\mathbb{Z}_2$ act homologically trivially and smoothly on a simply-connected, smooth 4-manifold $M$ that satisfies property $(\ast)$. Assume that $\mathrm{sign}(M)=0$. Then the action is pseudofree. In particular, every smooth, homologically trivial involution on a simply-connected, smooth, spin 4-manifold that satisfies property $(\ast)$ is pseudofree.
\end{cor}
\begin{proof}We have $c_1^2(M)=2\chi(M)+3\mathrm{sign}(M)$. We can take $s=-1$ in Proposition \ref{prop pseudofree} and the inequality is $0<3$, which is true. The second part follows from Ruberman's theorem \cite{Rub} since under these assumptions $\mathrm{sign}(M)=0$.
\end{proof}
\begin{rem} Atiyah-Bott \cite[Proposition 8.46]{AtBott} have shown that all components of the fixed point set have the same dimension, so that the fixed point set consists {\em either} of isolated fixed points {\em or} of a collection of embedded surfaces, if $\mathbb{Z}_2$ acts smoothly and orientation-preservingly on a simply-connected spin 4-manifold (there are generalizations to the locally linear and general case by Edmonds \cite[Corollary 3.3]{Ed} and Ruberman \cite{Rub}). Under our additional assumptions that the involution is homologically trivial and $M$ satisfies property $(\ast)$ the second case of a fixed point set of dimension 2 does not occur.
\end{rem}
We can prove a similar statement for $\mathbb{Z}_3$-actions on 4-manifolds close to or on the Bogomolov-Miyaoka-Yau line $c_1^2(M)=3\chi(M)$:
\begin{cor}\label{cor Z3 pseudofree} Let the group $\mathbb{Z}_3$ act homologically trivially and smoothly on a simply-connected, smooth 4-manifold $M$ that satisfies property $(\ast)$. Assume that either $c_1^2(M)=3\chi(M)-l$ with $0\leq l\leq 4$, or $M$ is minimal and $c_1^2(M)=3\chi(M)-l$ with $0\leq l\leq 8$. Then the action is pseudofree. 
\end{cor}
\begin{proof} The proof is similar to the proof of Corollary \ref{cor Z2 sign 0 -1}. For Proposition \ref{prop pseudofree} to work, $l$ has to be less than 6 in the first case and less than 10 in the second case.
\end{proof}
\begin{rem} Note that
\begin{equation*}
l=3\chi(M)-c_1^2(M)=\chi(M)-3\mathrm{sign}(M)=2-2b_2^+(M)+4b_2^-(M)
\end{equation*}
is always an {\em even} number. If $b_1(M)=0$, the Seiberg-Witten invariants can be non-zero or $M$ can have a symplectic form only if $b_2^+(M)$ is odd. Then $l$ is divisible by $4$. Hence if we want to apply Proposition \ref{prop sphere adjunction} and Proposition \ref{prop sphere adjunction b+=1}, then $l\in\{0,4\}$ in the first case and $l\in\{0,4,8\}$ in the second case of Corollary \ref{cor Z3 pseudofree}.
\end{rem}
\begin{ex}\label{ex Liu} Let $M$ be a smooth, minimal 4-manifold homeomorphic, but not diffeomorphic to the manifold $\mathbb{CP}^2\#2{\overline{\mathbb{CP}}{}^{2}}$, cf.~\cite{AkPark}. Suppose that $M$ admits a symplectic form $\omega$ (such an example for $M$ is constructed in that paper). Then $M$ satisfies property $(\ast)$ according to Proposition \ref{prop sphere adjunction b+=1}. Hence there does not exist a smooth, homologically trivial involution on $M$ and every smooth, homologically trivial $\mathbb{Z}_3$-action is pseudofree.
\end{ex}

\section{Actions on exotic $S^2\times S^2$ and $\mathbb{CP}^2\#{\overline{\mathbb{CP}}{}^{2}}$}\label{section exotic S2S2 CP2CP2bar}

\begin{lem} Let $\mathbb{Z}_p$, with $p\geq 3$ prime, act on $M$, where $M$ is a 4-manifold homeomorphic to $S^2\times S^2$ or $\mathbb{CP}^2\#{\overline{\mathbb{CP}}{}^{2}}$. Then the action is homologically trivial. 
\end{lem}
\begin{proof}
This follows as in \cite[Proposition 5.8]{Kl} (it follows from Lemma \ref{lem p b_2} in all cases except $p=3$).
\end{proof}
\begin{cor}\label{cor S2S2} Let $\mathbb{Z}_p$ act smoothly on $M$, where $M$ is a smooth, minimal 4-manifold homeomorphic, but not diffeomorphic to $S^2\times S^2$ or $\mathbb{CP}^2\#{\overline{\mathbb{CP}}{}^{2}}$ and satisfying property $(\ast)$. If $p=2$, assume in addition that the action is homologically trivial. Then the action is pseudofree.
\end{cor}
\begin{proof} We have $\chi(M)=4$ and $c_1^2(M)=8$. Hence the inequality in Proposition \ref{prop pseudofree} with $s=-2$ is
\begin{equation*}
4p-8< 4p-2.
\end{equation*}
Since this is true, the claim follows.
\end{proof}
Note that every smooth 4-manifold homeomorphic to $S^2\times S^2$ is minimal because its intersection form is even. It is not known if there exist exotic 4-manifolds homeomorphic, but not diffeomorphic to $S^2\times S^2$ or $\mathbb{CP}^2\#{\overline{\mathbb{CP}}{}^{2}}$. The smallest (in terms of Euler characteristic) known, simply-connected 4-manifold that admits exotic copies is $\mathbb{CP}^2\#2{\overline{\mathbb{CP}}{}^{2}}$, mentioned above in Example \ref{ex Liu}. However, if the trend for $\mathbb{CP}^2\#n{\overline{\mathbb{CP}}{}^{2}}$ with $n\geq 2$ generalizes to even smaller 4-manifolds, it is quite likely that exotic copies of $S^2\times S^2$ or $\mathbb{CP}^2\#{\overline{\mathbb{CP}}{}^{2}}$ exist, at least some of which could be symplectic, so that Corollary \ref{cor S2S2} applies to them.
\begin{rem} All statements in this paper remain true (except possibly Theorem \ref{Edmonds exist loc lin invol} of Edmonds) if the assumption that $M$ is simply-connected is replaced by $H_1(M;\mathbb{Z})=0$. This follows from \cite[Corollary 3.3, Proposition 3.5]{McC}, since in this situation Proposition \ref{main prop} above remains true. The results of Section \ref{section exotic S2S2 CP2CP2bar} then apply to smooth 4-manifolds with the integral cohomology of $S^2\times S^2$ and $\mathbb{CP}^2\#{\overline{\mathbb{CP}}}$ (for example, the symplectic cohomology $S^2\times S^2$ constructed in \cite{Ak}).
\end{rem}
\bibliographystyle{amsplain}

\begin{thebibliography}{999}

\bibitem{Ak} A.~Akhmedov, {\em Small exotic 4-manifolds}, Algebr.~Geom.~Topol.~{\bf 8}, 1781--1794 (2008). 

\bibitem{AkPark} A.~Akhmedov, B.~D.~Park, {\em Exotic smooth structures on small 4-manifolds with odd signatures}, Invent.~Math.~{\bf 181}, no.~3, 577--603 (2010).

\bibitem{AtBott} M.~F.~Atiyah, R.~Bott, {\em A Lefschetz fixed point formula for elliptic complexes. II. Applications}, Ann.~of Math.~{\bf 88}, 451--491 (1968).

\bibitem{AtSing} M.~F.~Atiyah, I.~M.~Singer, {\em The index of elliptic operators: III}, Ann.~of Math.~{\bf 87}, 546--604 (1968).

\bibitem{Br} G.~E.~Bredon, {\sl Introduction to compact transformation groups}, Academic Press, New York--London (1972).

\bibitem{BuRa} D.~Burns Jr., M.~Rapoport, {\em On the Torelli problem for k\"ahlerian $K-3$ surfaces}, Ann.~Sci.~\'Ecole Norm.~Sup.~(4) {\bf 8}, no.~2, 235--273 (1975). 

\bibitem{ChKw} W.~Chen, S.~Kwasik, {\em Symplectic symmetries of 4-manifolds}, Topology {\bf 46}, no.~2, 103--128 (2007). 

\bibitem{Don} S.~K.~Donaldson, {\em An application of gauge theory to four-dimensional topology}, J.~Differential Geom.~{\bf 18}, 279--315 (1983).

\bibitem{Ed1} A.~L.~Edmonds, {\em Construction of group actions on four-manifolds}, Trans.~Amer.~Math.~Soc.~{\bf 299}, no.~1, 155--170 (1987). 

\bibitem{Ed2} A.~L.~Edmonds, {\em Involutions on odd four-manifolds}, Topology Appl.~{\bf 30}, no.~1, 43--49 (1988).

\bibitem{Ed} A.~L.~Edmonds, {\em Aspects of group actions on four-manifolds}, Topology Appl.~{\bf 31}, no.~2, 109--124 (1989). 

\bibitem{GS} R.~E.~Gompf, A.~I.~Stipsicz, {\sl $4$-manifolds and Kirby calculus}, Graduate Studies in Mathematics, 20. Providence, Rhode Island. American Mathematical Society 1999.

\bibitem{Hirz} F.~Hirzebruch, {\em The signature theorem: reminiscences and recreation},  Prospects in mathematics (Proc. Sympos., Princeton Univ., Princeton, N.J., 1970), pp. 3--31, Ann.~of Math.~Studies, No.~70, Princeton Univ.~Press, Princeton, N.J.~1971. 

\bibitem{Kl} M.~Klemm, {\em Finite Group Actions on Smooth 4-Manifolds with Indefinite Intersection Form} (1995). Open Access Dissertations and Theses. Paper 2369. http://digitalcommons.mcmaster.ca/opendissertations/2369


\bibitem{Ko} D.~Kotschick, {\em Orientations and geometrisations of compact complex surfaces},  
Bull.~London Math.~Soc.~{\bf 29}, no.~2, 145--149 (1997).

\bibitem{KrMr} P.~B.~Kronheimer, T.~S.~Mrowka, {\em The genus of embedded surfaces in the projective plane}, Math.~Res.~Lett.~{\bf 1}, no.~6, 797--808 (1994).

\bibitem{KwSch} S.~Kwasik, R.~Schultz, {\em Homological properties of periodic homeomorphisms of 4-manifolds}, Duke Math.~J.~{\bf 58}, 241--250 (1989).

\bibitem{LiH} H.~Li, {\em Cyclic group actions on elliptic surfaces $E(2n)$}, J.~Math.~Comput.~Sci.~{\bf 2}, no.~6, 1759--1765 (2012). 

\bibitem{LiLiu} T.-J.~Li, A.~Liu, {\em Symplectic structure on ruled surfaces and a generalized adjunction formula}, Math.~Res.~Lett.~{\bf 2}, no.~4, 453--471 (1995).

\bibitem{Liu} A.-K.~Liu, {\em Some new applications of general wall crossing formula, Gompf's conjecture and its applications}, Math.~Res.~Lett.~{\bf 3}, no.~5, 569--585 (1996).

\bibitem{Mat} T.~Matumoto, {\em Homologically trivial smooth involutions on K3 surfaces}, Aspects of low-dimensional manifolds, 365--376, Adv.~Stud.~Pure Math., 20, Kinokuniya, Tokyo, 1992.

\bibitem{McC} M.~McCooey, {\em Symmetry groups of non-simply connected four-manifolds}, preprint: arXiv:0707.3835v2.

\bibitem{Pe} C.~A.~M.~Peters, {\em On automorphisms of compact K\"ahler surfaces}, Journ\'ees de G\'eometrie Alg\'ebrique d'Angers, Juillet 1979/Algebraic Geometry, Angers, 1979, pp. 249--267, Sijthoff \& Noordhoff, Alphen aan den Rijn--Germantown, Md., 1980.

\bibitem{PPX} U.~Persson, C.~Peters, G.~Xiao, {\em Geography of spin surfaces}, Topology {\bf 35}, 845--862 (1996).

\bibitem{Rub} D.~Ruberman, {\em Involutions on spin 4-manifolds}, Proc.~Amer.~Math.~Soc.~{\bf 123}, no.~2, 593--596 (1995).

\bibitem{Smith} P.~A.~Smith, {\em Transformations of finite period. IV. Dimensional parity}, 
Ann.~of Math.~(2) {\bf 46}, 357--364 (1945).

\bibitem{Sz} Z.~Szab\'o, {\em Exotic 4-manifolds with $b_2^+=1$}, Math.~Res.~Lett.~{\bf 3}, no.~6, 731--741 (1996).

\bibitem{Wall} C.~T.~C.~Wall, {\sl Surgery on Compact Manifolds}, Vol.~69 of Mathematical Surveys and Monographs, 2nd Edition, American Mathematical Society, Providence, RI, 1999

\bibitem{Wi} E.~Witten, {\em Monopoles and four-manifolds}, Math.~Res.~Lett.~{\bf 1}, no.~6, 769--796 (1994).

\end{thebibliography}

\bigskip
\bigskip

\end{document}